\documentclass[11pt]{article}
\usepackage[french, english]{babel}
\usepackage[utf8]{inputenc}
\usepackage{amsmath,amsfonts,amssymb,amsthm,mathrsfs}
\usepackage{minitoc}
\usepackage{bbm}
\renewcommand{\leq}{\leqslant}
\renewcommand{\geq}{\geqslant}
\usepackage[mono=false]{libertine}
\useosf
\usepackage{mathpazo}

\usepackage{wasysym}

\usepackage[dvipsnames]{xcolor}
\usepackage[a4paper,vmargin={3.5cm,3.5cm},hmargin={2cm,2cm}]{geometry}
\usepackage[font=sf, labelfont={sf,bf}, margin=1cm]{caption}
\usepackage{graphicx,graphics}
\usepackage{epsfig}
\usepackage{latexsym}
\usepackage{stmaryrd}
\usepackage{ae,aecompl}

\usepackage[english]{babel}
 \usepackage[colorlinks=true]{hyperref}
\usepackage{pstricks}
\usepackage{enumerate}
\usepackage{tikz}
\usepackage{fontawesome}
\usetikzlibrary{arrows,automata}

\renewcommand{\P}{\mathbb{P}}
\newcommand{\E}{\mathbb{E}}
\DeclareFixedFont{\beaupetit}{T1}{ftp}{b}{n}{2cm}

\newtheorem{theorem}{Theorem}[]
\newtheorem{definition}{Definition}

\newtheorem{proposition}[definition]{Proposition}

\newtheorem{lemma}[definition]{Lemma}

\theoremstyle{definition}

\newtheorem*{remark}{Remark}

\newcommand{\var}{\mathrm{Var}}
\newcommand{\eps}{\varepsilon}

\usepackage{todonotes}

\renewcommand{\H}{\mathbb{H}}

\tikzstyle{every node}=[circle, draw, fill=black!30, inner sep=0pt, minimum width=4pt]

\tikzstyle{texte}=[draw=none, fill=none]

\title{\textsc{On Cheeger constants of hyperbolic surfaces}}
\author{
Thomas \textsc{Budzinski}\thanks{ENS de Lyon.\hfill  \href{mailto:thomas.budzinski@ens-lyon.fr}{\texttt{thomas.budzinski@ens-lyon.fr}}}\qquad\&\qquad
Nicolas \textsc{Curien}\thanks{Universit\'e Paris-Saclay.\hfill  \href{mailto:nicolas.curien@gmail.com}{\texttt{nicolas.curien@gmail.com}}}
\qquad\&\qquad
Bram \textsc{Petri}\thanks{Sorbonne Universit\'e.\hfill  \href{mailto:bram.petri@imj-prg.fr}{\texttt{bram.petri@imj-prg.fr}}}}
\date{}
\begin{document}
\maketitle 

\vspace{-0.5cm}
\begin{abstract} It is a well-known result due to Bollobas that the maximal Cheeger constant of large $d$-regular graphs cannot be close to the Cheeger constant of the $d$-regular tree. We  prove analogously that the Cheeger constant of closed hyperbolic surfaces of large genus is bounded from above by $2/\pi \approx 0.63...$ which is strictly less than the Cheeger constant of the hyperbolic plane. The proof uses a random construction based on a Poisson--Voronoi tessellation of the surface with a vanishing intensity.\end{abstract}

\begin{figure}[!h]
 \begin{center}
 \includegraphics[height=4.5cm]{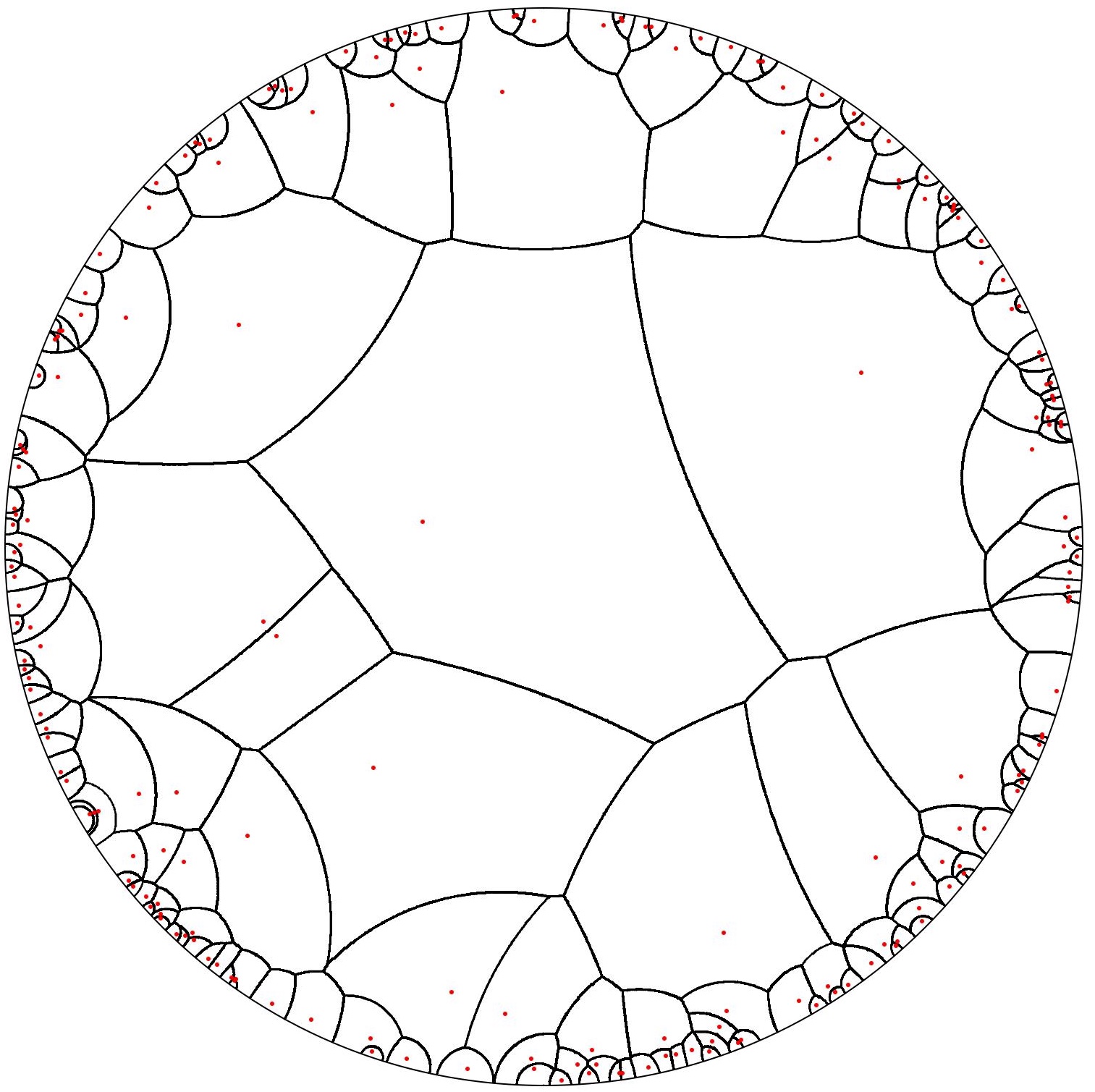} \hspace{0.5cm}
  \includegraphics[height=4.5cm]{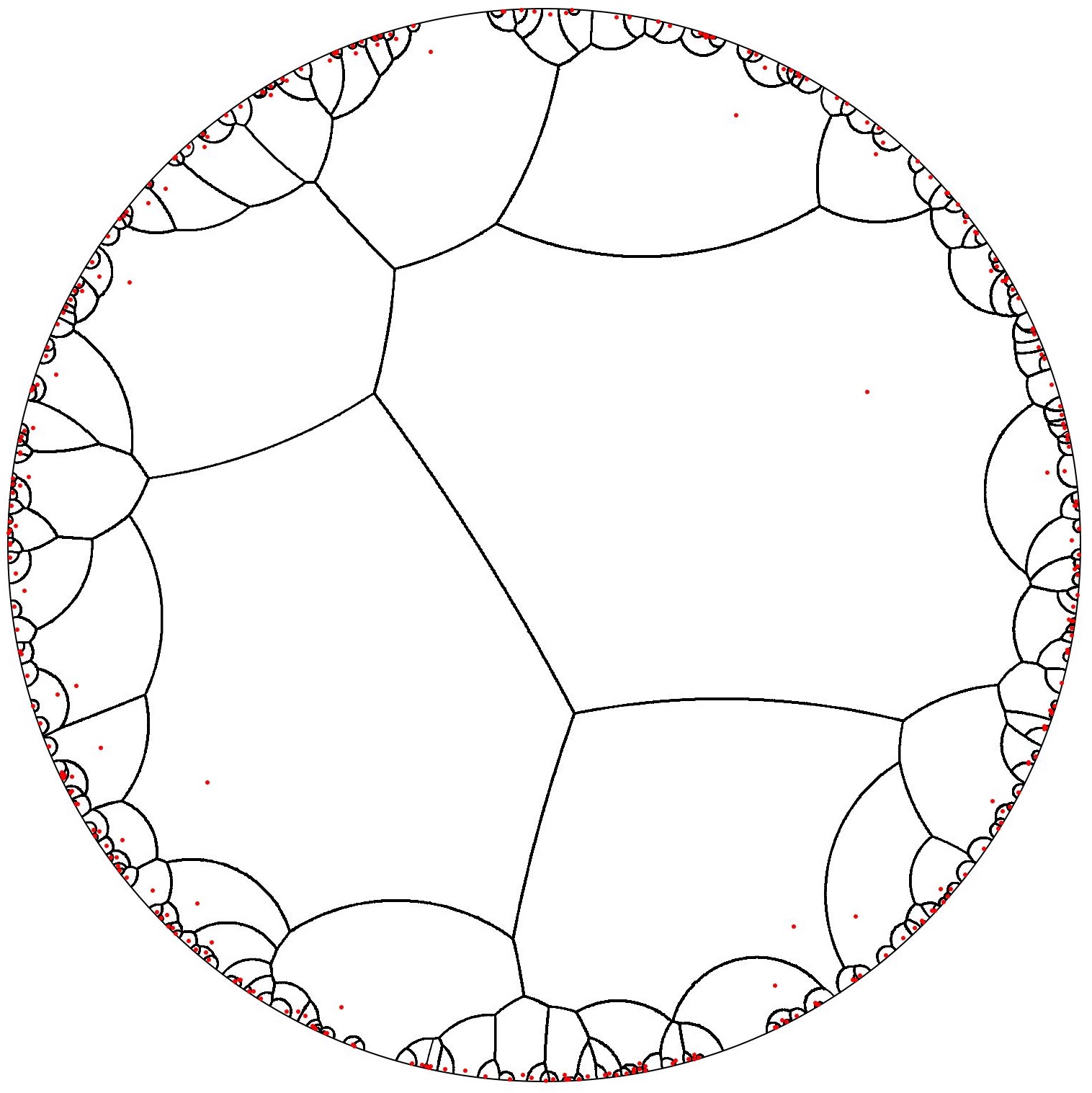} \hspace{0.5cm}
   \includegraphics[height=4.5cm]{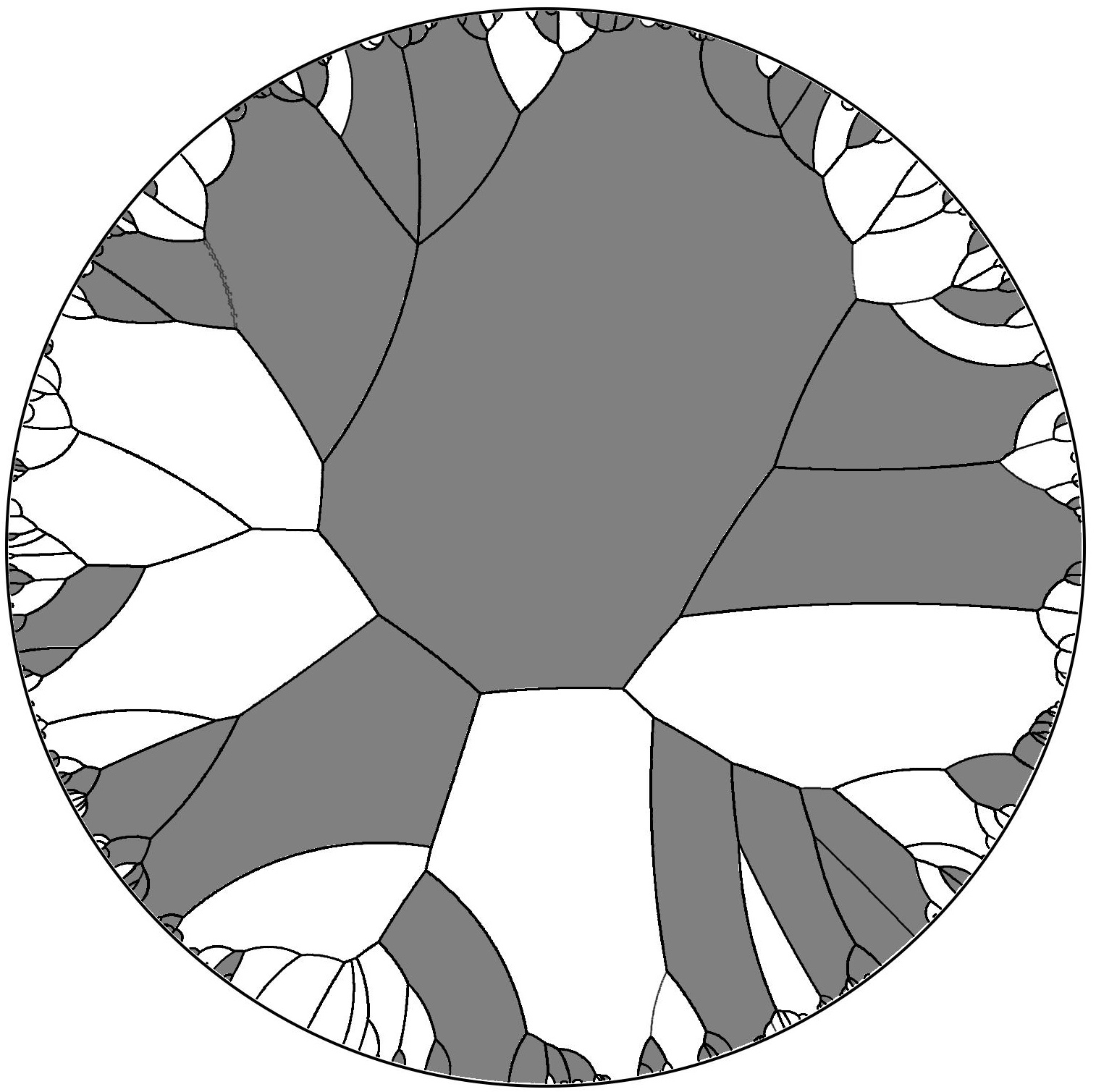}
 \caption{From left to right: Poisson--Voronoi tessellations of the hyperbolic plane with decreasing intensity. Their limit (on the right) is the  \emph{pointless Voronoi tessellation} of the hyperbolic plane whose cells have been colored in black/white uniformly at random. This object has an average "linear" density equal to $2 \times \frac{2}{\pi}$ per unit of area. \label{fig:voronoiH}}
 \end{center}
 \end{figure}

\vspace{-1cm}
\section{Introduction}
Let $ \mathcal{S}$ be a closed hyperbolic surface. If $A \subset \mathcal{S}$ is a subset of $ \mathcal{S}$, we define $ h^{*}(A) = |\partial A| / |A|$ where $| \partial A|$ is the length of its boundary and $|A|$ is its area. If one of these quantities is not well defined, then $h^{*}(A) = + \infty$ by convention. The \emph{Cheeger constant} of $ \mathcal{S}$ is 
 \begin{eqnarray} \label{eq:defcheeger} h( \mathcal{S}) = \inf \left\{h^{*}(A) | A \subset \mathcal{S} \mbox{ with } |A| \leq \, | \, \mathcal{S}|/2 \right\}.  \end{eqnarray}
\begin{theorem} \label{thm:main} For $g \geq 2$, let $ \mathcal{M}_{g}$ be the moduli space of all isometry classes of closed hyperbolic surfaces of genus $g$. Then we have 
$$ \limsup_{g \to \infty} \sup_{ \mathcal{S} \in \mathcal{M}_{g}} h( \mathcal{S}) \leq \frac{2}{\pi}.$$
\end{theorem}
We recall that the Cheeger constant of the hyperbolic plane $ \mathbb{H}$ is equal to $1$, which is asymptotically attained by large disks. As such, our result shows a gap between the maximal Cheeger constant of a large, closed hyperbolic surface and that of its universal cover. The presence of this gap was conjectured by Wright and Lipnowski \cite{Wright} building upon \cite{BrooksZuk} and inspired by similar results in graph theory. It also follows from our result that Cheeger's inequality \cite{Cheeger} -- the original reason for which the Cheeger constant was introduced -- cannot be used to show that a closed hyperbolic surface of large genus has an optimal spectral gap. Finally, Theorem \ref{thm:main} implies a recent result by Shen--Wu on random Bely\u{\i} surfaces \cite{ShenWu}.

\paragraph{Acknowledgment.} The authors are grateful to the organizers of the conference "Structures on surfaces" (CIRM), where this work was started.

\section{Context and sketch of proof}
Before sketching the proof of Theorem \ref{thm:main}, let us place it into context by telling the analogous story in the case of $d$-regular graphs. We will use similar notation as in the case of hyperbolic surfaces, so that the reader may follow the parallels more easily.

\subsection{Warm-up: the graph case} Fix an integer $d \geq 3$ and consider the set $ \mathcal{G}_d(n)$ of all connected simple graphs on $n$ vertices having all degree $d$ (we assume that $d\cdot n$ is even so that $\mathcal{G}_d(n)$ is not empty). If $ \mathfrak{g}_{n} \in \mathcal{G}_d(n)$ and $A$ is a subset of the vertices of $ \mathfrak{g}_{n}$, the isoperimetric constant of $A$ is defined as 
$$ h^*(A) = \frac{|\partial A|}{|A|},$$ where $|A|$ is the number of vertices of $A$ and $|\partial A|$ is the number of edges having one endpoint in $A$ and the other outside $A$. The Cheeger constant of $ \mathfrak{g}_{n}$ is then 
$$ h( \mathfrak{g}_{n}) := \inf \left\{ h^*(A) \, | \, A \subset  \mathfrak{g}_n \mbox{ with } |A| \leq n/2\right\}.$$ It is easy to construct graphs $ \mathfrak{g}_n \in \mathcal{G}_d(n)$ with small Cheeger constant, for example if they contain a large piece which looks roughly one-dimensional. On the other hand, we have $$  \mathbf{c}_d:= \liminf_{n \to \infty} \sup_{ \mathfrak{g}_{n} \in \mathcal{G}_d(n)} h( \mathfrak{g}_{n})  >0,$$ since there are  (families of) graphs, called \emph{expanders}, whose Cheeger constant is uniformly bounded from below.  The existence of such graphs has famously been proved by Margulis \cite{Margulis_explicit} through an explicit construction and Pinsker  \cite{Pinsker} using a probabilistic argument.  Ramanujan graphs \cite{lubotzky1988ramanujan,Margulis_ramanujan} are very good expanders and their existence shows that $ \mathbf{c}_d \geq   \frac{d}{2}- O( \sqrt{d})$ when $d \to \infty$, which can also be proved using random graphs \cite{Bollobas_isoperimetric}. Conversely, an easy argument shows that the Cheeger constant of a large $d$-regular graph is asymptotically bounded from above by that of the $d$-regular tree $ \mathbb{T}_d$, which is equal to $d-2$, so that in particular $ \mathbf{c}_d \leq d-2$. This bound is not sharp and in fact Bollobas \cite{Bollobas_isoperimetric}, later sharpened by Alon \cite{Alon}, proved that
\begin{eqnarray}
\label{eq:1/2} \mathbf{c}_d \leq  \frac{1}{2}(d-2) \quad  \mbox{for all } d \geq 3
\end{eqnarray}
(and even $ \mathbf{c}_d =  \frac{d}{2}- O (\sqrt{d})$ asymptotically as $d\to \infty$). This gap between the largest Cheeger constant of large $d$-regular graphs and that of $ \mathbb{T}_d$ is the graph analog of our Theorem \ref{thm:main}. \\
The idea of the proof of \eqref{eq:1/2} is surprisingly simple: fix a large graph $ \mathfrak{g}_{n} \in \mathcal{G}_d(n)$ and color \emph{uniformly at random} half of its vertices black and the other half white. Consider then the subset $A$ of size $n/2$ consisting of the black vertices. Since an edge is counted in $ |\partial A|$ if and only its endpoints have different colors, the probability that a given edge contributes to $| \partial A|$ is approximately $ \frac{1}{2}$. We thus have
\begin{equation}\label{eqn_graph_boundary_A}
\mathbb{E}[|\partial A|] \approx \frac{1}{2} | \mathrm{Edges}( \mathfrak{g}_{n})| = \frac{1}{2}\cdot\frac{dn}{2},
\end{equation}
and so $ h( \mathfrak{g}_{n}) \leq \mathbb{E}[h^*(A)] \leq \frac{d}{2}$ asymptotically when $n$ is large. To get the better bound $ h( \mathfrak{g}_{n}) \leq \frac{d-2}{2}$, the idea is to first regroup the vertices of $ \mathfrak{g}_{n}$ into connected ``regions'' $R_1, ... , R_k$ of vertices which are all large (i.e. $1 \ll |R_{i}|$), but not too large (i.e. $ |R_{i}| \ll n$). In particular $k$ must be large. In the graph case, one can perform such a splitting in various deterministic ways, e.g.~using a spanning tree of $ \mathfrak{g}_{n}$ as in \cite[Section 3]{Alon}. Since those regions are large and connected we have the crude bound valid on $d$-regular trees
 \begin{eqnarray} \label{eq:crudetree} h^{*}(R_{i}) \leq d-2 +o(1).  \end{eqnarray} Hence, the number of edges whose endpoints lie in two different regions $R_{i} \ne R_{j}$ is roughly $ \frac{n(d-2)}{2}$. We can then proceed as above and color each region uniformly at random in black or white. Since the regions are not too large, a standard concentration argument shows that  the black vertices form a subset $ \tilde{A}$ of approximately $n/2$ vertices. Again, a given edge is counted in $|\partial \tilde{A}|$ if its endpoints lie in two different regions with different colors. By the same computation as in~\eqref{eqn_graph_boundary_A} for $\E \left[ |\partial \tilde{A} |\right]$, we deduce $$ h^{*}( \tilde{A}) \preceq \frac{1}{2} \cdot \frac{n(d-2)}{2} \cdot \frac{2}{n} = \frac{d-2}{2}.$$

\subsection{Hyperbolic surfaces} Let us now draw the parallel with the case of hyperbolic surfaces. A first difficulty is that the \emph{a priori} crude bound  \eqref{eq:crudetree} does not hold in the continuous setting: there is no upper bound on the Cheeger constant of a connected set since there are such sets with a nasty fractal boundary having a large one-dimensional measure. However, the global inequality 
$$ h( \mathcal{S}_{g}) \leq h( \mathbb{H}) = 1+o(1) \quad \text{as }g\to\infty$$
is still true for any hyperbolic surface of genus $g\geq 2$. This can e.g.~be derived by spectral considerations, combining works of Cheeger \cite{Cheeger} and Cheng \cite{Cheng}. Here also, the existence of ``expander surfaces'' of large genus with Cheeger constant bounded away from $0$ is known. The known constructions essentially fall into three (overlapping) categories. First of all, there are multiple ways to compare the Cheeger constant of a hyperbolic surface to that of a graph \cite{Buser_cubicgraphs,Brooks_cheeger}. Secondly, just like in the case of graphs, it is known, due to Buser \cite{BuserNote}, that spectral expansion implies isoperimetric expansion. As such, the many examples of surfaces with a spectral gap -- the first examples are due to Selberg \cite{Selberg} and recently near optimal spectral expanders were found by Hide and Magee \cite{HideMagee} -- also give rise to surfaces with large Cheeger constants. Finally, multiple random constructions \cite{Mir,BM04} are known to provide examples.
 
Let us now try to mimic the proof of the graph case to prove a strict inequality in the last display.  Obviously, taking half of the points of $ \mathcal{S}_{g}$ does not make sense, and one wants to first split our deterministic surface $ \mathcal{S}_{g}$ into regions $ {R}_{1}, ... , {R}_{k}$ satisfying $1 \ll |{R}_{i}| \ll g$ before coloring them in black and white with equal probability. In \cite{BrooksZuk} Brooks and Zuk roughly speaking postulated the existence of such a decomposition based on balls where the isoperimetric constants $h^*(R_i)$ of the regions in question are close to $1$. Given those hypothetical decompositions, the coloring argument would yield 
 $$ h( \mathcal{S}_{g}) \preceq  \frac{1}{2}$$ when $g \to \infty$. We were not able to prove that such decompositions actually exist on every hyperbolic surface and rather use a \emph{random splitting of $ \mathcal{S}_{g}$ based on a Voronoi decomposition}. More precisely, we consider a Poisson point process with points $X_{1}, ... , X_{N}$ of intensity given by $\lambda  \cdot \mu_{ \mathcal{S}_g}$, where $ \mu_{ \mathcal{S}_g}$ is the hyperbolic area measure on $ \mathcal{S}_{g}$ with total mass $4\pi(g-1)$, and $\lambda >0$ is a small constant. Those points decompose $ \mathcal{S}_{g}$ into the Voronoi regions $ \mathrm{Vor}_{\lambda}( \mathcal{S}_{g}) := \{ C_{1}, ... , C_{N}\}$ where 
 $$ C_{i} = \left\{x \in \mathcal{S}_{g} : {d}_{\mathcal{S}_{g}}(x, X_{i}) = \min_{1 \leq j \leq N} {d}_{ \mathcal{S}_{g}}(x, X_{j}) \right\}.$$
When the intensity parameter $\lambda$ is small, the typical area of a region is of order $ 1/ \lambda$ and so one can expect (see Lemma \ref{lem_LLN_Voronoi_general} below) that indeed we have $ 1  \ll  | C_{i}| \ll | \mathcal{S}_{g}|$ at least for most regions. Now recall that we have no \emph{a priori} upper bound for the Cheeger constant of these regions (contrary to the graph case). The crux of the argument boils down to showing that, at least on average, we have $ h^{*}( C_{i}) = \frac{4}{\pi}$. More precisely, we will show (see Proposition \ref{prop:perimeter}) that the expected length of the union $\partial \mathrm{Vor}_{\lambda}( \mathcal{S}_{g})$ of the boundaries of the Voronoi cells $C_{1}, \dots , C_{N}$ satisfies
 \begin{eqnarray} \label{eq:boundarypointless}  \limsup_{\lambda \to 0} \sup_{g \geq 2}\sup_{ \mathcal{S}_{g} \in \mathcal{M}_{g}}  \frac{1}{| \mathcal{S}_g|}\mathbb{E}\left[\left|\partial \mathrm{Vor}_{\lambda}( \mathcal{S}_{g})\right|\right]   \leq   \frac{2}{\pi}.  \end{eqnarray}
 Given the above display, one can then run the same proof as in the graph case: we color uniformly the Voronoi cells in black and white and consider the resulting black component. Its volume is concentrated around $\frac{1}{2}| \mathcal{S}_{g}|$ and its boundary size is less than $| \mathcal{S}_g| \cdot \frac{1}{2} \cdot (\frac{2}{\pi} + \varepsilon)$ for $g$ large when $\lambda$ is small, so that $ h( \mathcal{S}_{g}) \leq (\frac{2}{\pi} + \varepsilon)$ asymptotically as claimed in Theorem \ref{thm:main}. 
 
 \subsection{The pointless Voronoi tessellation of $ \mathbb{H}$}
Let us gain some intuition on the proof of \eqref{eq:boundarypointless} done in Proposition \ref{prop:perimeter}. Suppose first that the systole of $ \mathcal{S}_{g}$ tends to $\infty$ as $ g \to \infty$. In that situation, the neighborhood of each point in $  \mathcal{S}_g$ looks like a piece of the hyperbolic plane and  the Voronoi tessellation $ \mathrm{Vor}_{\lambda}( \mathcal{S}_{g})$ converges in distribution (in the local Hausdorff sense) towards the Voronoi tessellation $ \mathrm{Vor}_{\lambda}(  \mathbb{H})$ of the hyperbolic plane with intensity $\lambda$, see Figure \ref{fig:voronoiH}.

 This classical object has been studied in stochastic geometry \cite{isokawa2000some,calka2021poisson} and in particular in relation to its percolation properties \cite{benjamini2001percolation,hansen2021poisson,hansen2022critical}. In particular, Isokawa \cite{isokawa2000some} computed the mean characteristics of a typical\footnote{By Palm calculus, such a cell can be obtained by adding the point $0$ to the Poisson process and considering the associated region.} cell $C_{\lambda}$ in $ \mathrm{Vor}_{\lambda}( \mathbb{H})$ : this cell is an almost surely finite convex hyperbolic polygon satisfying 
 $$ \mathbb{E}[ C_{\lambda}] = \frac{1}{\lambda} \quad \mbox{ and } \quad \mathbb{E}[ |\partial C_{\lambda}|] =  \frac{8}{\sqrt{\pi \lambda}} \int_{0}^{\infty} \mathrm{e}^{-u} \sqrt{u + \frac{u^{2}}{4\pi \lambda}} \, \mathrm{d}u.$$
In particular, the ``average Cheeger constant'' of cells of $ \mathrm{Vor}_{\lambda}( \mathbb{H})$ satisfies in the small intensity limit
$$ \frac{  \mathbb{E}[ |\partial C_{\lambda}|]}{ \mathbb{E}[ | C_{\lambda}|]}  \xrightarrow[\lambda \to 0]{}   \frac{4}{\pi},$$
which explains \eqref{eq:boundarypointless} in the case when $\mathrm{Systole}(\mathcal{S}_g)\to\infty$. Note that recently percolation on hyperbolic Poisson--Voronoi tessellation with small intensity has been studied in \cite{hansen2021poisson}. Underneath the convergence of the above display as the intensity tends to $0$ lies the fact that $ \mathrm{Vor}_{\lambda}( \mathbb{H})$ converges (in distribution for the Hausdorff topology on compact sets of $ \mathbb{H}$) towards a limiting object that we name the \emph{pointless Poisson--Voronoi tessellation} of the hyperbolic disk (see Figure \ref{fig:voronoiH}) and whose construction and properties will be studied in a forthcoming work.

Finally, in order to deal with the case in which the systole is not large, we will need to study $\partial \mathrm{Vor}_{\lambda}( \mathcal{S}_g)$ in the neighbourhood of a point $x$ of $\mathcal{S}_g$. To do so,  we will replace the hyperbolic plane by the \emph{Dirichlet domain} -- a specific fundamental domain --  of $\mathcal{S}_g$ around $x$. We will prove that in order to prove \eqref{eq:boundarypointless}, it will be enough to understand $\partial \mathrm{Vor}_{\lambda}( \mathcal{S})$ inside the Dirichlet domain, thus also reducing the general case to computations in the hyperbolic plane.

\section {Proof}
Let us recall some basic notions and fix notation before starting the proof.

\subsection{Preliminaries}
\paragraph{The hyperbolic plane.}
We recall that the \emph{hyperbolic plane} $\H$ is the unit disk $\{ z \in \mathbb{C} :  \big| |z|<1 \}$, equipped with the Riemannian metric $\frac{4 |\mathrm{d}z|^2}{(1-|z|^2)^2}$. This metric has constant curvature $-1$ and is invariant under M\"obius transformations. We will denote by $d_{\H}(x,y)$ the hyperbolic distance in $\H$. For several computations in the paper, it will be useful to use polar coordinates on $\H$. More precisely, for $r>0$ and $\theta \in [0,2\pi)$, we denote by $[r;\theta]$ the point of $\H$ of the form $z=\rho(r) \mathrm{e}^{i\theta}$, where $\rho(r)>0$ is such that $d_{\H}(0,z)=r$. We note that the area measure on $\H$ can be written as $\sinh(r) \, \mathrm{d}r \, \mathrm{d}\theta$.

\paragraph{Hyperbolic surfaces and Dirichlet domains.}
We recall that a closed hyperbolic surface $ \mathcal{S}$ of genus $g\geq 2$ is a Riemannian surface which is locally isometric to $\H$, or equivalently which is the quotient of $\H$ by a discrete, torsion-free, cocompact group $G$ of isometries. We denote by $d_{ \mathcal{S}}$ and $\mu_{ \mathcal{S}}$ (or just $d$ and $\mu$ if no confusion is possible) the hyperbolic metric and area measure on $ \mathcal{S}$. If $A \subset  \mathcal{S}$ is a Borel subset we write $|A|$ for $\mu(A)$ and $|\partial A|$ for the length of its boundary. If $x$ is a point of a hyperbolic surface and $r \geq 0$, we denote by $B_r(x)$ the closed ball of radius $r$ around $x$.

We will make important use of the \emph{Dirichlet domain}, which is a particular choice of a fundamental domain for the action of $G$ on $\H$, see \cite[Section 9.4]{Beardon}. More precisely, let $x$ be a point of a hyperbolic surface $ \mathcal{S}=\H/G$, and let $p:\H \to  \mathcal{S}$ be its universal cover, so that $p(0)=x$. We denote by $D( \mathcal{S},x)$ the set of those points $y' \in \H$ such that $d_{\H}(0,y')=d_{\mathcal{S}}(x,p(y'))$. In other words, this means that the image under $p$ of the geodesic from $0$ to $y'$ is still a shortest path in $\mathcal{S}$. Dirichlet domains are convex polygons of $ \mathbb{H}$ \cite[Theorem 9.4.2]{Beardon}. For (almost) every point $y \in S$, we will denote by $p^{-1}(y)$ the unique point of $D(\mathcal{S},x)$ which is sent to $y$. Note that the maps $p:D(\mathcal{S},x) \to \mathcal{S}$ and $p^{-1}:\mathcal{S} \to D(\mathcal{S},x)$ are measure-preserving.

\paragraph{Poisson point processes.} The surface $ \mathcal{S}$ or the hyperbolic plane $ \mathbb{H}$ both carry a Borel measure $ \mu_{  \mathcal{S}}$  which is of finite mass $4\pi(g-1)$ in the case of $ \mathcal{S}$ and $\mu_{ \mathbb{H}}$ which is $\sigma$-finite in the case of $ \mathbb{H}$. They have no atoms, so one can define a Poisson point process on those spaces with intensity $\lambda \cdot \mu_{ \cdot}$. This is a cloud of distinct random points that is characterized by the fact that for any disjoint measurable subsets $A_{1}, ..., A_{k}$, the number of points falling inside $A_{i}$ are independent Poisson random variables of mean $|A_i|$, see \cite{kingman1992poisson} for details.

\paragraph{Poisson--Voronoi tessellation.} Let $ \mathcal{S}$ be a hyperbolic surface of (large) genus $g$. As announced above, to bound its Cheeger constant from above, we shall build a random subset $A_{\lambda}$ of $\mathcal{S}$ in the following way. We first throw a Poisson point process $\Pi = \{X_1, ... , X_N\}$ on $\mathcal{S}$ with intensity $\lambda \cdot \mu_{ \mathcal{S}}$, where $\lambda >0$ is a small constant. In particular the random variable $N$ follows a Poisson law with mean $\lambda \cdot 4 \pi  (g-1)$. We then consider the closed Voronoi cells $ \mathrm{Vor}_\lambda(  \mathcal{S})= \{C_1, ... , C_N\}$ it defines  and denote by $ C(x)$ the Voronoi cell containing the point $x \in \mathcal{S}$ (with ties broken arbitrarily).
Conditionally on $\Pi$, each cell is colored black or white independently with probability $\frac{1}{2}$ and we let $A_\lambda \subset \mathcal{S}$ be (the closure of) the union of the black cells.
On the one hand, we will prove (Lemma \ref{lem_LLN_Voronoi_general}) that unless $h( \mathcal{S})$ is very small (in which case our main result is trivial), the area $|A_\lambda|$ is close to $\frac{|\mathcal{S}|}{2}$ if the surface is large enough. On the other hand, we will show (Proposition \ref{prop:perimeter}) that $\E \left[ \left| \partial A_{\lambda} \right|\right]$ is close to $\frac{1}{\pi} | \mathcal{S}|$ if the intensity $\lambda$ is chosen small enough. 

\subsection{Area estimate}
We start with the area estimate. We believe that Lemma~\ref{lem_LLN_Voronoi_general} below should be true even without the Cheeger constant assumption, but we could not find a short argument to prove the general statement.

\begin{lemma}[Area estimate]\label{lem_LLN_Voronoi_general}
For any $\lambda>0$ and $\delta>0$, there is $g_0\geq 2$ with the following property. For every hyperbolic surface $ \mathcal{S}$ of genus $g \geq g_0$ such that $h( \mathcal{S}) \geq \delta$, if the random subset $A_{\lambda}$ of $\mathcal{S}$ is built as described above, we have
	\[ \P \left( \left| \frac{|A_{\lambda}|}{| \mathcal{S}|}-\frac{1}{2} \right| > \delta \right)<\delta. \]
\end{lemma}

\begin{proof}
	Let $ \mathcal{S}$ be a hyperbolic surface with Cheeger constant at least $ \delta$. It follows immediately from our setup that $\E \left[ |A_\lambda| \right] = \frac{|\mathcal{S}|}{2}$, so we only need to bound the variance of $|A_\lambda|$. By conditioning on $ \mathrm{Vor}_\lambda( \mathcal{S})$, we have
	\begin{align*}
	\var \left( |A_\lambda| \right) &=   \mathbb{E}\left[\frac{1}{4} \sum_{i=1}^N | C_i|^2\right] = \frac{1}{4} \int_{\mathcal{S}^2} \P \left( C(x)=C(y) \right) \, \mu(\mathrm{d}x) \, \mu(\mathrm{d}y).
	\end{align*}
	To bound $\P \left( C(x)=C(y) \right)$ from above, we will first argue that the assumption $h( \mathcal{S}) \geq \delta$ implies a lower bound on the volume of balls around $x$ and $y$ for "most" points $x,y \in S$.
	
	More precisely, for all $x \in S$ and $r>0$, by the Cheeger constant assumption we have
	\[ \frac{\mathrm{d}}{\mathrm{d}r} \left| B_r(x) \right| = \left| \partial B_r(x) \right| \geq \delta \left| B_r(x) \right|.\]
	Therefore, let $r_1>r_0>0$ (the values of $r_0$ and $r_1$, depending only on $\delta$ and $\lambda$, will be specified later), and assume that $g$ is large enough to have $2\pi(\cosh(r_1)-1)<\frac{| \mathcal{S}|}{2}$. We have
	\[ \left| B_{r_1}(x) \right| \geq  \mathrm{e}^{\delta(r_1-r_0)} \left| B_{r_0}(x) \right|.\]	
	On the other hand, by the collar lemma, we know that there is an absolute constant $C>0$ such that, if $r_0$ is small enough, we have
	\[
	\left| \left\{ x\in  \mathcal{S} \, | \, \mathrm{InjRad}(x) \leq r_0 \right\} \right| \leq C r_0^2 \cdot | \mathcal{S}|.
	\]
	But if the injectivity radius around $x$ is larger than $r_0$, then $\left| B_{r_0}(x) \right| = 2\pi \left( \cosh(r_0)-1 \right)$ and we get a lower bound on $\left| B_{r_1}(x) \right|$. In particular, by taking $r_0$ small enough, we find
	\begin{equation}\label{eqn_few_points_with_small_balls}
	\left| \left\{ x\in  \mathcal{S} \, | \, \left| B_{r_1}(x) \right|< \mathrm{e}^{\delta(r_1-r_0) 2\pi} \left( \cosh(r_0)-1 \right) \right\} \right| \leq \delta^3 | \mathcal{S}|.
	\end{equation}
	We denote by $K_{r_1} \subset  \mathcal{S}^2$ the set of pairs $(x,y)$ such that $d(x,y)>2r_1$ and neither $x$ nor $y$ satisfies the event in the left-hand side of~\eqref{eqn_few_points_with_small_balls}. Let $(x,y) \in K_{r_1}$. If $x$ and $y$ belong to the same Voronoi cell, since $B_{r_1}(x)$ and $B_{r_1}(y)$ are disjoint, at least one of these two balls contains none of the points $X_i$. It follows that
	\begin{align*}
	\P \left( C(x)=C(y) \right) & \leq \exp \left( -\lambda |B_{r_1}(x)| \right) + \exp \left( -\lambda |B_{r_1}(y)| \right)\\
	&\leq 2\exp \left( -2\pi \lambda  \mathrm{e}^{\delta(r_1-r_0)} \left( \cosh(r_0)-1 \right)   \right).
	\end{align*}
	In particular, if we have chosen a large enough value for $r_1$, this probability is smaller than $\delta^3$. Therefore, we get
	\[ \var \left( |A_\lambda| \right) \leq \frac{\delta^3| \mathcal{S}|^2}{4} + \frac{\left|  \mathcal{S}^2 \backslash K_{r_1} \right|}{4} \leq \frac{\delta^3| \mathcal{S}|^2}{4}+2\frac{\delta^3 | \mathcal{S}|^2}{4} + \frac{2\pi}{4} \left( \cosh(2r_1)-1 \right) | \mathcal{S}|,\]
	where the second term comes from~\eqref{eqn_few_points_with_small_balls}, and the third term counts the pairs $(x,y)$ with $d(x,y)<2r_1$. In particular, for $| \mathcal{S}|= 4\pi(g-1)$ large enough, the variance of $|A_\lambda|$ is smaller than $\delta^3| \mathcal{S}|^2$ and the conclusion follows by the Chebychev inequality.
\end{proof}

\subsection{Perimeter estimate}
Let us now pass to the perimeter estimate which is the most technical part of the proof.
\begin{proposition}[Perimeter estimate] \label{prop:perimeter} Recall that $ \partial \mathrm{Vor}_\lambda( \mathcal{S})$ is the union of the sides of the Poisson-Voronoi cells $C_1, ... , C_N$ with intensity $\lambda \cdot \mu_ { \mathcal{S}}$. Then we have 
$$ \limsup_{\lambda \to 0} \sup_{g \to \infty} \sup_{ \mathcal{S} \in \mathcal{M}_g} \frac{1}{| \mathcal{S}|}\mathbb{E}\left[\left|\partial  \mathrm{Vor}_\lambda ( \mathcal{S})\right| \right] \leq \frac{2}{\pi}.$$
\end{proposition}

\paragraph{An optimization lemma.}
Before proving this estimate, we state two intermediate results that will be useful for us. The first is a nice optimization lemma, for which a very short and direct proof is provided in the last page of~\cite{Toth56}.

\begin{lemma}\label{lem_exo_colle}
	Let $\nu$ be a probability measure on $[0,2\pi]$. Then we have
	\[ \int_0^{2\pi} \int_0^{2\pi} \left| \sin \frac{\theta_1-\theta_2}{2} \right| \nu(\mathrm{d}\theta_1) \nu(\mathrm{d}\theta_2) \leq \frac{2}{\pi},\]
	with equality if $\nu$ is the uniform measure.
\end{lemma}

The proof of Lemma~\ref{lem_exo_colle} consists of assuming by density that $\nu$ has a smooth density $f$ with respect to Lebesgue, and expressing the integral in terms of the Fourier coefficients of $f$. We will use this lemma to handle the fact that a Dirichlet domain $D(\mathcal{S},x)$ is not as isotropic as $\H$.

\paragraph{Intersection of thin rings.}
The second lemma will help ruling out some pathological behaviours of Dirichlet domains. For $0<a<b$ and $x \in \H$, we denote by $R_a^b(x)$ the set of points $z \in \H$ such that
\[ a \leq d(x,z) \leq b. \]

\begin{lemma}\label{lem_thinrings}
	Let $K>0$ and let $x,y$ be two distinct points of $\H$. Then for all $a>0$, we have
	\[ \left| R_a^{a+\eps}(x) \cap R_a^{a+\eps}(y) \right| = o(\eps) \]
	as $\eps \to 0$. Moreover, the $o(\eps)$ is uniform in $(x,y,a)$ provided $a \leq K$ and $d(x,y) \geq K^{-1}$.	
\end{lemma}

\begin{proof}
	We first note that if $d(x,y)>2K+2$, then the intersection $R_a^{a+\eps}(x) \cap R_a^{a+\eps}(y)$ is empty as soon as $\eps<1$, so we may assume $d(x,y) \leq 2K+2$. For the same reason, since $d(x,y) \geq K^{-1}$, we may assume $a \geq \frac{1}{2K}$. Moreover, by invariance under M\"obius transformations, we may assume that $x=0$ and that $y=[d;0]$ in polar coordinates, with $K^{-1} \leq d \leq 2K+2$.
	
	Let us express $R_a^{a+\eps}(0) \cap R_a^{a+\eps}(y)$ in polar coordinates. First, if $z=[r;\theta] \in R_a^{a+\eps}(0) \cap R_a^{a+\eps}(y)$, then we must have $r \in [a,a+\eps]$. Second, by the hyperbolic law of cosines, we have
	\[ \cosh(d(y,z)) = \cosh(d) \cosh(r) - \cos(\theta) \sinh(d) \sinh(r). \]
	Since $a \leq d(y,z) \leq a+\eps$, we deduce
	\[  \frac{\cosh(d) \cosh(r) - \cosh(a+\eps)}{\sinh(d) \sinh(r)} \leq \cos(\theta) \leq \frac{\cosh(d) \cosh(r) - \cosh(a)}{\sinh(d) \sinh(r)}. \]
	That is, $\cos (\theta)$ lies in an interval $I(a,d,r,\eps)$ of length $\frac{\cosh(a+\eps)-\cosh(a)}{\sinh(d) \sinh(r)} \leq K \frac{\eps \sinh(a+\eps)}{\sinh(r)} \leq C(K) \eps$ by the assumption $r \geq \frac{1}{2K}$. This implies that $\theta$ must lie in a set $S(a,d,r,\eps) \subset [0,2\pi]$ of measure at most $C(K) \sqrt{\eps}$. Therefore, we have
	\begin{align*}
	\left| R_a^{a+\eps}(x) \cap R_a^{a+\eps}(y) \right| &= \int_a^{a+\eps} |S(a,d,r,\eps)| \sinh(r) \, \mathrm{d}r\\
	&\leq \sinh(a+\eps) C(K) \eps^{3/2} \leq \sinh(K+1) C(K) \eps^{3/2},
	\end{align*}
	which proves the lemma.
\end{proof}

\begin{proof}[Proof of Proposition~\ref{prop:perimeter}] In the rest of the proof we write  $ \partial C \equiv \partial \mathrm{Vor}_\lambda( \mathcal{S})$ to lighten notation. For all $\eps>0$, we denote by $\partial^{\eps} C$ the set of points of $\mathcal{S}$ lying at hyperbolic distance at most $\eps$ from $ \partial C$. Since $\partial C$ is a.s. the union of finitely many geodesic segments, we have $$ |\partial C| = \lim_{ \varepsilon \to 0} \frac{1}{2\eps} |\partial^{\eps}C|.$$	It follows from  Fatou's lemma that
	\begin{equation}\label{eqn_thickening_boundary}
	\E \left[ |\partial C| \right] \leq \liminf_{\eps \to 0} \frac{1}{2\eps} \E \left[ |\partial^{\eps}C| \right] = \liminf_{\eps \to 0} \frac{1}{2\eps} \int_{\mathcal{S}} \P \left( x \in \partial^{\eps}C \right) \, \mu_{\mathcal{S}}(\mathrm{d}x).
	\end{equation}
Hence, let $0<\eps<1$ and $x \in \mathcal{S}$, and let $i_1$ be such that that $C(x)=C_{i_1}$. We first note that if $x \in \partial^{\eps} C$, then there are at least two Voronoi cells which intersect the ball $B_{\eps}(x)$, and one of these cells is $C_{i_1}$. Hence, there is an $i_2\ne i_1$ such that the bisector between $X_{i_1}$ and $X_{i_2}$ intersects $B_{\eps}(x)$. Therefore we introduce, for any point $y$, the set $A^{\eps}_{\mathcal{S}}(x,y) \subset \mathcal{S}$ consisting of those  points $z$ such that $d(x,z) \geq d(x,y)$ and such that the bisector between $y$ and $z$ intersects the closed ball $B_{\eps}(x)$. We first note that if
	$z\in A^{\eps}_ \mathcal{S}(x,y)$,
	then by the triangle inequality, we have\footnote{Using this inequality to crudely bound $|A^{\eps}_{\mathcal{S}}(x,y)|$, we would obtain Proposition~\ref{prop:perimeter} (and therefore Theorem~\ref{thm:main}) with a constant $1$ instead of $\frac{2}{\pi}$. This is why we will need the more accurate description given by Lemma~\ref{lem_description_A_domain}.}
	\begin{equation}\label{eqn_couronne}
	d(x,y) \leq d(x,z) \leq d(x,y)+2\eps.
	\end{equation}
	The event $x \in \partial^{\eps}C$ is equivalent to saying that at least one other point of the Poisson process lands in the region $A^{\eps}_{\mathcal{S}}(x,X_{i_1})$. Therefore, by conditioning on the point $X_{i_1}$ closest to $x$, we have
	\begin{align}\label{eqn_proba_near_bdry_as_integral}
	\P \left( x \in \partial^{\eps}C \right) &= \lambda \int_ \mathcal{S}  \mu_{\mathcal{S}}(\mathrm{d}y) \  \exp \left(- \lambda |B_{d(x,y)}(x)|\right) \times \left( 1-\exp \left( -\lambda |A^{\eps}_ \mathcal{S}(x,y)| \right) \right)  \\
	&\leq \lambda^2 \int_\mathcal{S} \mu_{\mathcal{S}}(\mathrm{d}y) \ \exp \left(- \lambda |B_{d(x,y)}(x)|\right)  |A^{\eps}_\mathcal{S}(x,y)|.
	\end{align}
If the injectivity radius at $x\in \mathcal{S}$ is much larger than $d(x,y)$, then $|A^{\eps}_ \mathcal{S}(x,y)|$ can be computed as in the hyperbolic plane (and one would recover the estimates of Isokawa \cite{isokawa2000some}). In the general case, we will work with the Dirichlet domain  $D(\mathcal{S},x) \subset \H$ defined above. We recall that $D(\mathcal{S},x)$ is a fundamental domain for the projection $p : \H \to S$ with $p(0)=x$. 
We claim that for $y \in \mathcal{S}$, the subset $A^{\eps}_ \mathcal{S}(x,y)$ of $\mathcal{S}$ is very close to the subset $A^{\eps}_{D( \mathcal{S},x)}(0,p^{-1}(y))$ of $\mathbb{H}$.
	
	\begin{lemma}\label{lem_A_surface_and_domain}
		For $x,y \in  \mathcal{S}$, we have
		\[ \left| A^{\eps}_\mathcal{S}(x,y) \right| \leq \left| A^{\eps}_{D(\mathcal{S},x)}(0,p^{-1}(y)) \right| + o(\eps) \]
as $\eps \to 0$, uniformly in $x,y \in \mathcal{S}$.
	\end{lemma}

	\begin{proof}
		Let us fix $x$ and $y$, and write $r=d_\mathcal{S}(x,y)$. Since $p^{-1}:\mathcal{S} \to D(\mathcal{S},x)$ is measure-preserving, it is sufficient to prove
		\begin{equation}\label{eqn_bad_set_Dirichlet}
		\left| p^{-1} \left( A^{\eps}_\mathcal{S}(x,y) \right) \backslash A^{\eps}_{D(\mathcal{S},x)}(0,p^{-1}(y)) \right| = o(\eps),
		\end{equation}
		uniformly in $(x,y) \in \mathcal{S}^2$. We fix a Fuchsian group $G$ such that $\mathcal{S}=G\backslash\H$. Moreover, to avoid heavy notation involving the projection map, we will use $w'\in D(\mathcal{S},x)$ to denote the unique (up to a set of measure $0$) pre-image of $w\in\mathcal{S}$ under $p$. In particular, $x'=0$.
		
		Now assume that $\eps < \frac{1}{2} \mathrm{Systole}(\mathcal{S})$ so that $B_{\eps}(0) \subset D(\mathcal{S},x)$, and let $z'$ be in the difference of sets of~\eqref{eqn_bad_set_Dirichlet}. This implies that there is a point $a\in B_{\eps}(x)$ such that
		$$
		d_{\mathcal{S}}(a,z) < d_{\mathcal{S}}(a,y) \quad \text{but} \quad d_{\H}(a',z') > d_{\H}(a',y').
		$$
		Now let $g \cdot a'$ (for some $g \in G$) be the translate of $a'$ such that $d_{\mathcal{S}}(a,z)=d_{\H}(g \cdot a', z')$. The last display means that $d_{\H}(g\cdot a',z') < d_{\H}(a',y)$. Observe that $g \cdot a'\in B_\eps(g\cdot x')$. We get
		\begin{equation*}
		d_{\H}(g\cdot x',z') \leq d_{\H}(g\cdot x',g\cdot a') +   d_{\H}(g\cdot a',z')  \leq \eps + d_{\mathcal{S}}(a,z) \leq 2\eps +d_{\mathcal{S}}(x,z) \leq r+4\eps,
		\end{equation*}
  	where we have used \eqref{eqn_couronne} for the last inequality. On the other hand, by the definition of the Dirichlet domain, we have
  	$$
		d_{\H}(g\cdot x',z') \geq d_{\H}(x',z') = d_{\mathcal{S}}(x,z) \geq d_{\mathcal{S}}(x,y)=r.
	$$
	Finally, we have 
	$$
	d_{\H}(x',g\cdot x') \leq d_{\H}(x',z') + d_{\H}(z',g\cdot x') = d_{\mathcal{S}}(x,z) + d_{\H}(z',g\cdot x')  \leq 2r+6\eps,
	$$
	where we have used \eqref{eqn_couronne} again.
	Putting the last three inequalities together, we find that the set described in the left-hand side of~\eqref{eqn_bad_set_Dirichlet} is contained in
		\[ \bigcup_{\substack{w\in G\cdot 0 \setminus\{0\} \\ d_{\H}(0,w) \leq 2r+1}} \left\{ z' \in \H | \begin{array}{c} r \leq d(w,z') \leq r+4\eps \text{ and} \\ r \leq d(0,z') \leq r + 2\eps \end{array} \right\}.\]
	The sets in this union are intersections of two annuli of width $4\eps$, centered around $0$ and a translate $w$ of $0$. Such an intersection has area $o(\eps)$, by Lemma \ref{lem_thinrings}. Note that in order to apply this lemma, we use the fact that $r\leq \mathrm{Diameter}(\mathcal{S})$ and that a translate $w = g\cdot 0$ satisfies $d(w,0) \geq \mathrm{Systole}(\mathcal{S})$.
	
	Finally, the number of sets in the union can be bounded in terms of $r$, and hence in terms of $\mathrm{Diameter}(\mathcal{S})$. Indeed, two points in $G\cdot 0$ are least $\mathrm{Systole}(S)$ apart and as such only so many of them fit in a disk of radius $2r+1$. This concludes the proof of Lemma~\ref{lem_A_surface_and_domain}.
	\end{proof}

	The next step is to give a precise description of the set $A^{\eps}_{D(\mathcal{S},x)} (0,y)$. It will be particularly natural to express this description in polar coordinates. For all $r>0$, we denote by $I_r(x)$ the set of angles $\theta \in [0,2\pi)$ such that the point $[r;\theta]$ belongs to $D(\mathcal{S},x)$. We note that $I_r(x)$ is a finite union of intervals, and that $I_{r_2}(x) \subset I_{r_1}(x)$ when $r_1 \leq r_2$ by convexity of $D(\mathcal{S},x)$.
Moreover, the area measure on $D(\mathcal{S},x)$ can be written as $\sinh(r) \mathbbm{1}_{\theta \in I_r(x)} \, \mathrm{d}r \, \mathrm{d}\theta$. Then we have the following good approximation of $A^{\eps}_{D(\mathcal{S},x)}(0,y)$ (see also Figure~\ref{fig_A_on_Dirichlet_domain}).

\begin{figure}[!h]
 \begin{center}
 \includegraphics[height=5cm]{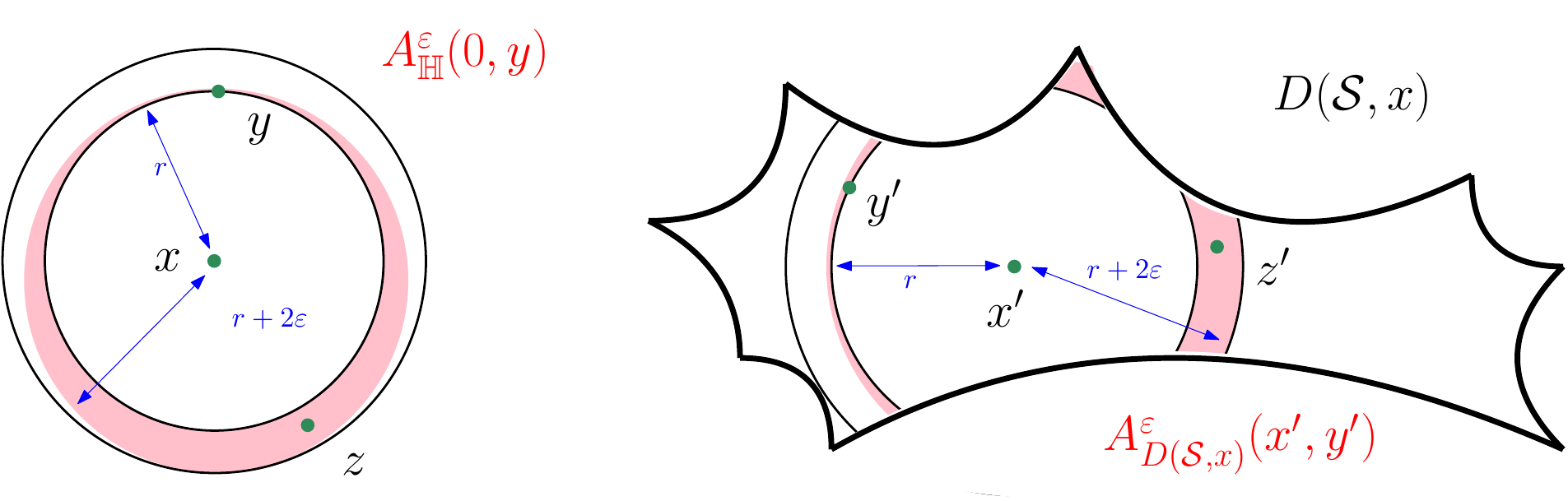}
 \caption{On the left, in pink the set $ A^{{ \varepsilon}}_{ \mathbb{H}}(x,y)$, see Lemma \ref{lem_description_A_domain}. On the right the set $ A^{{ \varepsilon}}_{ D( \mathcal{S},x)}(x',y')$ which by Lemma \ref{lem_A_surface_and_domain} is a very good approximation of $ A^{ \varepsilon}_{ \mathcal{S}}(x,y)$. \label{fig_A_on_Dirichlet_domain}}
 \end{center}
 \end{figure}	
	
	\begin{lemma}\label{lem_description_A_domain}
		Let $\delta>0$. Then there are $r_0(\delta)>1$ and $\eps_0(\delta)>0$ with the following property. For all $0<\eps<\eps_0(\delta)$ and any point $y = [r;\theta] \in \H$ such that $d(0,y)=r>r_0(\delta)$, we have the inclusion
		\begin{equation}\label{eqn_A_inclusion_planar}
		A^{\eps}_{\H}(0,y) \subset \left\{ [r';\theta'] \in \H | r \leq r' \leq r+2(1+\delta) \left| \sin \frac{\theta'-\theta}{2} \right| \eps \right\}.
		\end{equation}
		In particular, under those assumptions, if $x \in \mathcal{S}$ and $y \in D(\mathcal{S},x)$, we have
		\[ \left| A^{\eps}_{D(\mathcal{S},x)}(0,y) \right| \leq 2(1+\delta)\eps \sinh(r+3\eps) \int_{I_r(x)} \left| \sin \frac{\theta'-\theta}{2} \right| \, \mathrm{d}\theta'.\]
	\end{lemma}

\begin{proof}
	This is just a calculation using hyperbolic trigonometry. We write $z=[r';\theta'] \in A^{\eps}_{\H}(0,y)$. We want to prove that $r' \leq r+\left( 2(1+\delta) \left| \sin \frac{\theta'-\theta}{2} \right| \right) \eps$.
	Because $z \in A^{\eps}_{\H}(0,y)$, there is a point $a$ of the form $[\eps;\varphi]$ with $\varphi \in [0,2\pi)$ which is closer to $z$ than to $y$. Using the hyperbolic cosine law, we can write down the distances $d(a,y)$ and $d(a,z)$ in terms of $r$, $r'$, $\eps$, $\theta$, $\theta'$ and $\varphi$. We find
	\[ \cosh(r') \cosh(\eps) - \cos(\varphi-\theta') \sinh(r') \sinh(\eps) \leq \cosh(r) \cosh(\eps) - \cos(\varphi-\theta) \sinh(r) \sinh(\eps), \]
	or equivalently
	\begin{equation}\label{eqn_find_phi}
	\cos(\varphi-\theta') \sinh(r')-\cos(\varphi-\theta) \sinh(r) \geq \coth(\eps) \left( \cosh(r')-\cosh(r) \right) \geq \frac{\cosh(r')-\cosh(r)}{\eps}.
	\end{equation}
	Moreover, by the Cauchy--Schwarz inequality, we have
	\begin{align*}
	& \cos(\varphi-\theta') \sinh(r')-\cos(\varphi-\theta) \sinh(r)\\
	&= \left( \cos (\theta') \sinh(r') - \cos(\theta) \sinh(r) \right) \cos(\varphi) +\left( \sin (\theta') \sinh(r') - \sin(\theta) \sinh(r) \right) \sin(\varphi)\\
	&\leq \sqrt{\left( \cos (\theta') \sinh(r') - \cos(\theta) \sinh(r) \right)^2 + \left( \sin (\theta') \sinh(r') - \sin(\theta) \sinh(r) \right)^2}\\
	&= \sqrt{\sinh^2(r')+\sinh^2(r)-2\cos(\theta'-\theta) \sinh(r) \sinh(r')}\\
	&= \sqrt{\left( \sinh(r')-\sinh(r) \right)^2 + 2(1-\cos (\theta'-\theta)) \sinh(r) \sinh(r')}.
	\end{align*}
	Hence~\eqref{eqn_find_phi} squared becomes
	\[ \eps^2 \left( \sinh(r')-\sinh(r) \right)^2 + 4 \eps^2 \sin^2 \frac{\theta'-\theta}{2} \sinh(r) \sinh(r') \geq \left( \cosh(r')-\cosh(r) \right)^2. \]
	If $\eps$ is small enough and $r$ large enough (depending only on $\delta$), the first term in the left-hand side is small compared to the right-hand side, so we have
	\[ 4 \eps^2 \sin^2 \frac{\theta'-\theta}{2} \sinh(r) \sinh(r') \geq \left( 1+\delta/2 \right)^{-2} \left( \cosh(r')-\cosh(r) \right)^2, \]
	which becomes
	\[2\eps \left| \sin \frac{\theta'-\theta}{2} \right| \geq \left( 1+\frac{\delta}{2} \right)^{-1} \frac{\cosh(r')-\cosh(r)}{\sqrt{\sinh(r) \sinh(r')}} \geq \left( 1+\frac{\delta}{2} \right)^{-1} \frac{(r'-r) \sinh(r)}{\sqrt{\sinh(r) \sinh(r')}},\]
	which implies
	\[ r'-r \leq \left( 1+\frac{\delta}{2} \right) 2\eps \left| \sin \frac{\theta'-\theta}{2} \right| \sqrt{\frac{\sinh(r')}{\sinh(r)}}. \]
	Finally, recalling $r' \leq r+2\eps$, if $r$ is large enough (depending only on $\delta$), this last expression is smaller than $(1+\delta)2\eps \left| \sin \frac{\theta'-\theta}{2} \right|$, which concludes the proof of~\eqref{eqn_A_inclusion_planar}.
	
	Let us move on to the second point. We know that $A^{\eps}_{D(\mathcal{S},x)}(0,y) \subset D(\mathcal{S},x) \cap A^{\eps}_{\H}(0,y)$. Using the expression of the area measure in polar coordinates, Equation~\eqref{eqn_A_inclusion_planar} translates into
	\begin{align*}
	\left| A_{D(\mathcal{S},x)}^{\eps}(0,y) \right| &\leq \int_0^{2\pi} \int_r^{r+2(1+\delta) |\sin \frac{\theta'-\theta}{2}|\eps} \mathbbm{1}_{\theta' \in I_{r'}(x)} \sinh(r') \, \mathrm{d}r' \, \mathrm{d}\theta'\\
	&\leq \int_{I_r(x)} \int_r^{r+2(1+\delta) |\sin \frac{\theta'-\theta}{2}|\eps} \sinh(r') \, \mathrm{d}r' \, \mathrm{d}\theta'\\
	&\leq 2(1+\delta)\eps \sinh(r+3\eps) \int_{I_r(x)} \left| \sin \frac{\theta'-\theta}{2} \right| \, \mathrm{d}\theta',
	\end{align*}
	where the second inequality uses the inclusion $I_{r'}(x) \subset I_r(x)$. This concludes the proof of Lemma~\ref{lem_description_A_domain}.
\end{proof}

	We can now finish the proof of Proposition~\ref{prop:perimeter}. We re-express~\eqref{eqn_proba_near_bdry_as_integral} as an integral over $D(\mathcal{S},x)$ (since the projection $p$ preserves the measure), and write it down in polar coordinates:
	\begin{align*}
	\P \left( x \in \partial^{\eps} C \right) &\leq \lambda^2 \int_0^{+\infty} \exp \left( -\lambda |B_r(x)| \right) \int_{I_r(x)} \left| A^{\eps}_\mathcal{S}(x,p([r;\theta])) \right| \sinh(r) \, \mathrm{d}\theta \, \mathrm{d}r \\
	&= o(\eps) + \lambda^2 \int_0^{+\infty} \exp \left( -\lambda |B_r(x)| \right) \int_{I_r(x)} \left| A^{\eps}_{D(\mathcal{S},x)}(0,[r;\theta]) \right| \sinh(r) \, \mathrm{d}\theta \, \mathrm{d}r,	
	\end{align*}
	where the $o(\eps)$ is uniform in $x$, and the last part comes from Lemma~\ref{lem_A_surface_and_domain} and the fact that $D(\mathcal{S},x)$ is bounded. We now assume that $\eps$ is smaller than the $\eps_0(\delta)$ of Lemma~\ref{lem_description_A_domain}. For $r$ larger than the $r_0(\delta)$ of Lemma~\ref{lem_description_A_domain}, we bound $\left| A^{\eps}_{D(\mathcal{S},x)}(0,[r;\theta]) \right|$ using Lemma~\ref{lem_description_A_domain}. For $r \leq r_0(\delta)$, we use the crude bound $\left| A^{\eps}_{D(\mathcal{S},x)}(0,[r;\theta]) \right| \leq 4\pi \eps \sinh(r+2\eps)$ coming from~\eqref{eqn_couronne}. We obtain
	\begin{align*}
		\P \left( x \in \partial^{\eps} C \right) &\leq o(\eps) + \lambda^2 \int_0^{r_0(\delta)}  \int_{0}^{2\pi} 4\pi \eps \sinh(r+2\eps) \sinh(r) \, \mathrm{d}\theta \, \mathrm{d}r\\
		&+ 2 \lambda^2 (1+\delta) \eps \int_{r_0}^{+\infty} \exp \left( -\lambda |B_r(x)| \right) \int_{I_r(x)^2} \left| \sin \frac{\theta'-\theta}{2} \right| \sinh(r) \sinh(r+3\eps) \, \mathrm{d}\theta \, \mathrm{d}\theta' \mathrm{d}r.
	\end{align*}
	The first integral is bounded by $C(\delta)\lambda^2 \eps$, so if $\lambda$ is chosen smaller than some $\lambda_0(\delta)$ it is smaller than $\delta \eps$. Moreover, up to increasing the value $r_0(\delta)$, we may assume $\sinh(r+3\eps) \leq (1+\delta) \sinh(r)$. By these remarks and Lemma~\ref{lem_exo_colle} to handle the integral over $I_r(x)^2$, we find
	\begin{equation}\label{eqn_proba_boundary_as_integral_in_r}
	\P \left( x \in \partial^{\eps} C \right) \leq o(\eps)+\delta \eps + \frac{4}{\pi} \lambda^2 (1+\delta)^2 \eps \int_{r_0}^{+\infty} \exp \left( -\lambda |B_r(x)| \right) \left|I_r(x) \right|^2 \sinh^2(r) \, \mathrm{d}r.
	\end{equation}
	Our goal is now to write this in a form which can be directly integrated. For this, we notice that
	\[ |I_r(x)| \sinh(r) = \left| \partial B_r(x) \right| = \frac{\mathrm{d}}{\mathrm{d}r} \left| B_r(x) \right|. \]
	Therefore, for all $r \geq r_0(\delta)$, we have
	\[  \left| B_r(x) \right| = \int_0^r |I_s(x)| \sinh(s) \, \mathrm{d}s \geq |I_r(x)|(\cosh(r)-1). \]
	Up to increasing the value $r_0(\delta)$, we may assume $\cosh(r)-1 \geq (1+\delta)^{-1} \sinh(r)$, so that $|I_r(x)| \sinh(r) \leq (1+\delta)|B_r(x)|$. Therefore, replacing one of the two factors $|I_r(x)| \sinh(r)$ in~\eqref{eqn_proba_boundary_as_integral_in_r}, we obtain
	\begin{align*}
	\P \left( x \in \partial^{\eps} C \right) &\leq o(\eps)+\delta \eps + \frac{4}{\pi} \lambda^2 (1+\delta)^3 \eps \int_{0}^{+\infty} \left( \frac{\mathrm{d}}{\mathrm{d}r} \left| B_r(x) \right| \right) \left| B_r(x) \right| \exp \left( -\lambda \left| B_r(x) \right| \right)  \, \mathrm{d}r\\
	&= o(\eps)+\delta \eps + \frac{4}{\pi} \lambda^2 (1+\delta)^3 \eps \left[ -\left( \frac{1}{\lambda^2} + \frac{\left| B_r(x) \right|}{\lambda} \right) \exp \left( -\lambda \left| B_r(x) \right| \right) \right]_{0}^{+\infty}\\
	&= o(\eps)+\delta \eps + \frac{4}{\pi} (1+\delta)^3 \eps,
	\end{align*}
	where the $o(\eps)$ is still uniform in $x$. Plugging this into~\eqref{eqn_thickening_boundary}, we obtain
	\[ \E \left[ |\partial C| \right] \leq \frac{2}{\pi} (1+\delta)^4 |\mathcal{S}| \]
	for $\lambda<\lambda_0(\delta)$, which finally proves Proposition~\ref{prop:perimeter}.
 \end{proof}
	
\begin{remark}
	Let us briefly compare our approach for Proposition~\ref{prop:perimeter} to the computations of Isokawa~\cite{isokawa2000some} in the hyperbolic plane. The main difference is that Isokawa takes a point of the Poisson process as the center of polar coordinates, whereas we center them around a typical point of $\mathcal{S}$. In particular, Isokawa does not need to consider an $\eps$-thickening of the boundary. However, if we tried to adapt the arguments of~\cite{isokawa2000some} to our Dirichlet domains, we would need to carefully study the interplay between the sets $I_r(x)$ and $I_{s}(x)$ for some $r \ne s$, which seems complex. On the other hand, our argument only uses the interplay between $I_r(x)$ and itself (this is the role of Lemma~\ref{lem_exo_colle}).
\end{remark}

\subsection{Proof of Theorem \ref{thm:main}}	
Let $ 0 < \delta < \frac{1}{2} < \frac{2}{\pi}$ and let $\mathcal{S}$ be a surface with genus $g$ and Cheeger constant larger than $\delta$ (otherwise our result is trivial). Let $\lambda>0$ be small enough so that we have by Proposition \ref{prop:perimeter}
$$\E \left[ |\partial  \mathrm{Vor}_{\lambda}( \mathcal{S})| \right] \leq \frac{2}{\pi} (1+ \delta) |\mathcal{S}|.$$
Recall that $A_{\lambda}$ is obtained by coloring each cell of $ \mathrm{Vor}_{\lambda}( \mathcal{S})$ with probability $1/2$ independently. In particular, each side of $ \partial \mathrm{Vor}_{\lambda}( \mathcal{S})$ is in $\partial A_{\lambda} $ with probability $1/2$ (conditionally on $ \mathrm{Vor}_{\lambda}( \mathcal{S})$), so 
	\[\E \left[ |\partial A_{\lambda}| \right] \leq  \frac{1}{\pi} (1+ \delta) |\mathcal{S}|.\]
	Therefore, by the Markov inequality
	\[ \P \left( |\partial A_{\lambda}| \leq \frac{1}{\pi}(1+\delta)^2  | \mathcal{S}| \right) \geq 1-\frac{(1+\delta) | \mathcal{S}|/\pi}{(1+\delta)^{2} | \mathcal{S}|/\pi}=\frac{\delta}{1+\delta}.\]
	Now suppose  that $g$ is large enough so that Lemma~\ref{lem_LLN_Voronoi_general} holds with $\delta$ replaced by $\delta^{2}$. In this case,  since $ \frac{\delta}{1+\delta} + (1- \delta^{2})>1$ there is a positive probability that both $(1-\delta) \frac{| \mathcal{S}|}{2} \leq |A_{\lambda}| \leq (1+\delta) \frac{| \mathcal{S}|}{2}$ and $|\partial A_{\lambda}| \leq \frac{1}{\pi}(1+\delta)^2 | \mathcal{S}|$. This implies
	\[ h( \mathcal{S}) \leq \frac{(1+\delta)^2}{1-\delta} \frac{2}{\pi}, \]
	which proves Theorem~\ref{thm:main}. Et voil\`a.

\bibliographystyle{siam}
\bibliography{bib_hyperbolic.bib}

\end{document}